\documentclass[12pt, leqno]{amsart}

\usepackage{setspace}
\usepackage{graphicx}
\usepackage{amsfonts, amsmath, oldgerm, amssymb, amscd, bbm, amsthm}
\usepackage{dsfont}
\usepackage[T1]{fontenc}

\newtheorem{theorem}{Theorem}[section]
\newtheorem{lemma}[theorem]{Lemma}

\newtheorem{corollary}[theorem]{Corollary}

\newcommand{\Z}{\mbox{$\mathbb Z$}}
\newcommand{\C}{\mbox{$\mathbb C$}}
\newcommand{\Q}{\mbox{$\mathbb Q$}}
\newcommand{\A}{\mbox{$\mathbb A$}}
\newcommand{\ol}{\overline}
\newcommand{\p}{\mbox{$\mathbb P$}}
\newcommand{\sk}{\ol \kappa}
\newcommand{\s}{\widetilde{S}}

\title{$\Q$-Homology Planes pairs with $\sk=1$}
\author{Sagar Kolte}

\begin{document}
\let\thefootnote\relax

\begin{abstract} 
A pair $(S,C)$ is called a singular $\Q$-homology plane pair if  $S$ is a singular projective surface with only quotient singularities having the same rational homology as $\p^2$ 
and $C \subset S$ has the same rational homology as $\p^1$. We will prove results concerning smooth rational curves on $S$ and the singularities of $S$. Such that $\sk(S^0)=1$ and $\sk(S-C) \neq -\infty$. We end with an example of such pairs.
\end{abstract}

\maketitle
\section{Introduction}
In this paper all varieties will be over $\C$, the field of complex numbers. By $\sk$ we denote the logarithmic Kodaira dimension (\cite{iitaka82}[Chapter 11]) of a smooth open algebraic surface.\\

Let $S$ be a projective surface with only quotient singularities such that $H_i(S, \Q)=H_i(\C\p^2, \Q)$ for all $i \in \Z$. Let $S^0$ be the smooth locus of $S$. Let $C$ be a rational curve in $S^0$ with only cuspidal singularities. It is easy to see that $S-C$ is an $\A^2-\Q$-Homology Plane. The pair $(S,C)$ is called a $\Q$-homology plane pair.\\

\begin{theorem}
If $\ol \kappa(S^0-C)=1$ and $\ol \kappa(S-C) \neq -\infty$ then $C$ has at most two cusps and there is a smooth rational curve $\theta \subset S^0$ passing through the cusps of $C$. 
\end{theorem}

As a consequence of this we have:\\

\begin{corollary}
The following hold:
\begin{enumerate}
\item $S-\theta$ is a $\Z$-homology plane, 
\item $S$ has at most one singularity,
\item If $\pi : S' \rightarrow S$ is a resolution of the singularity of $S$ and $E=\pi^{-1}(Sing(S)) \subset S'$, then $E$ is an irreducible rational curve.
\end{enumerate}
\end{corollary}

It might be noted that the links of Quotient singular points of surfaces have a seifert fibration. It is of considerable interest(see \cite{kollar08}[Section 3]) to understand the H-cobordisms of seifert fibrations. Theorem 1.1 and Corollary 1.2 imply that if $\sk(S^0-C)=1$, then there is another curve $\theta$ such that $S-(p \cup \theta)$ is a H-cobordism of the link of the singularity $p$ of $S$.\\

The Corollary finds a more technical proof in another preprint by the author and R.V. Gurjar and D-S. Hwang. Theorem 1.1 substantially simplifies the proof and addresses the question of H-cobordisms mentioned above.

By a result of Y. Kawamata, the hypothesis $\ol \kappa(S^0-C)=1$ implies that $S^0-C$ admits a unique $\C^*$-fibration. This naturally leads us to consider $\C^*$-fibrations on $\Q$-homology planes. In this context the following result(Lemma 2.10, \cite{miya91}) is useful:

\begin{lemma}[Miyanishi-Sugie]
Let $X$ be an affine $\Q$-homology plane with a $\C^*$-fibration, $f:X \rightarrow B$. Then we have:
\begin{enumerate}
\item $B$ is either $\p^1$ or $\A^1$. If the fibration is twisted then $C \cong \A^1$.
\item If $B \cong \p^1$ then every fiber of $f$ is irreducible and there is exactly one fiber isomorphic to $A^1$.
\item if $B \cong A^1$ and $f$ is untwisted, all fibers are untwisted except one which consists of two irreducible components. If $\C \cong A^1$ and $f$ is twisted, all fibers are irreducible and there is exactly one fiber isomorphic to $\A^1$.
\end{enumerate}

\end{lemma}

We setup some notation. $\pi: \tilde{S} \rightarrow S$ is a resolution of singularities of $S$ and $C$. Assume that $\pi^{-1}(C)$ is SNC-minimal and $\pi^{-1} ( Sing(S))$ is minimal. Let $D:= \pi^*(C)$ and $E:= \pi^{-1}(Sing(S))$.\\

In the rest of this work we will use the terminology developed by T. Fujita(see \cite{fujita82}[$\S$.3]) in describing divisors on a surface.\\

By hypothesis $C$ is a rational curve in $S^0$ with cuspidal singularities. If $C$ is uni-cuspidal then we have the following description (see \cite{orevkov02}) of $D$: $D= B \cup C' \cup T_2 \cup T_3$. Here $C'$ is the proper transform of $C$, $B$ is a branch of $D$ with branching number 3. The curves $C'$, $T_2$, $T_3$ are branches of $B$ where $T_2$ is a linear chain of rational curves, $T_3$ is a tree of rational curves such that the irreducible components of $T_3$ have branching number at most 3. Further $B \cup T_2 \cup T_3$ can be blown down to a smooth point. $B$ is the only curve in $D$ with self intersection -1.\\

\section{Uni-Cuspidal Rational Curves}
In this section we prove some results about uni-cuspidal rational curves such that $C' \cdot C' \geq 1$.

\begin{lemma}
If $C$ is uni-cuspidal and  $C' \cdot C'=1$ then $C$ is smooth.
\end{lemma}

\begin{proof}
Let $h: H \rightarrow \tilde{S}$ be a blow up at a smooth point of $D$ on $C'$. The proper transform of $C'$ under $h$ will be denoted by $C'_h$. Clearly $C'_h \cdot C'_h=0$, hence $C'_h$ is a full fiber of a $\p^1$-fibration $f$ on $H$. Let $l$ be the (-1) curve in the total transform of $D$ under $h$. Let $D_h$ denote the total transform of $D$ under $h$. Then clearly $H-D_h$ has a $\C^*$-fibration which we denote by $f$. As $T_2$ and $T_3$ both meet $B_h=h^{-1}(B)$, which is a section of $f$ on $H$, they lie in distinct fibers of $f$. Let $F_2$ and $F_3$ be the respective fibers. The fiber $F_2$ meets $l$, as $l$ is a section of $f$. Let $l_1$ denote the irreducible component of $F_2$ which meets $l$. As $l_1$ is reduced, there is another irreucible component of $F_2$, which is a (-1) curve. We denote this by $M$. Both $l_1$ and $M$ are not in $D_h$, hence the singular fiber of the $\C^*$-fibration $f$ corresponding to $F_2$ is reducible. Similarly, we can argue that $F_3 \cap {H-D_h}$ is reducible. This contradicts (3) of Lemma 1.3. Hence $C=C'$ and $C$ is smooth.
\end{proof}
Let $C$ be a uni-cuspidal rational curve. By $n(C)$ denote the minimum number of blow-ups required so that the total transform of $C$ is SNC.

\begin{lemma}
Let the notations be as above. Let $C' \cdot C' \textgreater 0$ and $C$ be singular. Then there exists a rational curve $\theta$ with cuspidal singularities in $S^0$ such that $n(C) \textgreater n(\theta)$. Also, $\theta$ passes through the cusps of $C$.
\end{lemma}

\begin{proof}
We have assumed $C$ to be singular, thus by Lemma 2.1, $C' \cdot C' \textgreater 1$. Let $h: H \rightarrow \tilde{S}$ denote a sequence of blowing up of infinitely near points of a smooth point of $D$ on $C'$ so that $C'_h \cdot C'_h=0$. The exceptional locus of $h$ is a linear chain of rational curves with self intersections: $[1,2,...,2]$. Let the (-1) curve in the exceptional locus be denoted by $L$. As in Lemma 2.1, $f$ is a $\p^1$ fibration on $H$ and $C'_h$ is a full fiber of $f$. The curves $B_h$ and $L$ are the disjoint sections of $f$. The chain of (-2) curves in the exceptional locus of $h$ is in a fiber $F$ of $f$. By the notation above, $T_2$ is also a linear chain of rational curves. We claim that both $T_2$ and $[2,2,...,2]$ lie in $F$. To prove this claim, first note that $F \cap D_h$ has at least two connected components. If $F \cap D_h$ had exactly one connected component (i.e. $[2,...,2]$), then as in Lemma 2.1 $F \cap (H-D_h)$ would be reducible and if $F_2$ and $F_3$ are the singular fibers containing $T_2$ and $T_3$, then $F_i \cap (H-D_h)$ would also be reducible. This contradicts Lemma 1.3.\\

If both $T_2$ and $T_3$ are linear, the claim can be proved by interchanging $T_2$ and $T_3$. Assume therefore that $T_3$ is not linear and $T_3 \subset F$ and $[2,...,2] \subset F$. There is a curve $J$ in $F$ adjacent to $[2,...,2]$ which is not contained in $D_h$. Assume first that this is a (-1) curve. We blow it down and the subsequent (-1) curves which arise from the $[2,...,2]$ chain. As there is a reduced irreducible component in the image of $T_3$ in this sequence of blow downs, there is another (-1) curve in $F$ meeting $T_3$. This makes $F \cap (H-D_h)$ reducible. When $J$ is not a (-1) curve, there is already another curve in $F$ which is a (-1) curve and which is not contained in $D_h$. This makes $F \cap (H-D_h)$ reducible. On the other hand the fiber $T_2 \subset F_2$ is also such that $F_2 \cap (H-D_h)$ is reducible. This contradicts Lemma 1.3. Thus $T_2$ and $[2,...,2]$ are both in $F$ and clearly $F \cap (H-D_h) \cong \C^*$. It is easy to see that the curve $\theta$ in $S$, of which the (-1) curve $\theta'= F-(T_2 \cup [2,...,2])$ is the proper transform, is the required curve in the statement of the lemma such that $\theta'' \cdot \theta'' \textgreater 0$. Also note that $T_2$ is irreducible.
\end{proof}

Note that if $\theta$ is smooth then $\theta'$ is at the tip of the linear chain of rational curves $T_2 \cup B \cup T_3$, hence $\theta$ meets $C$ transversally. This gives
\begin{corollary}
Let the notation and hypothesis be as in Lemma 2.2, then there is a smooth rational curve $\theta \subset S^0$. Also, $\theta$ passes through the cusp of $C$ and meets $C$ transversally.
\end{corollary}

\section{Proof of Theorem 1.1}
We prove the Theorem in several steps.

\begin{lemma}
If $\ol \kappa(\tilde{S}-D)=1$ then $C$ has at most two cusps.
\end{lemma}
\begin{proof}
Because $\ol \kappa(\tilde{S}-D)=1$ there is a $\C^*$ fibration $f$ on $\tilde{S}-D$. Assume if possible that $C$ has three or more cusps. Then it is clear that $D$ has at least three (-1) curves each with three branches and none of them are adjacent to each other.These (-1) curves arise from the final blow up to resolve each cusp and make $D$ a SNC divisor.\\

We claim that none of these (-1) curves can be contained in a fiber: Indeed if a (-1) curve $l$ from $D$ is contained in a fiber $F$ of $f$ then it cannot have branching number three in $F$. Thus at least one of the curves adjacent to $l$ is horizontal to $f$. But this makes $l$ a reduced (-1) curve in $F$. Hence $l$ cannot have branching number two. Hence two curves adjacent to $l$ are horizontal to $f$. None of them are (-1) curves, hence the other (-1) curves in $D$ are in the fibers of $f$ because there are only two components of $D$ which are horizontal to $f$. Let $l'$ be the other (-1) curve in $D$ which is contained in a fiber $F'$ of $f$. By a similar reasoning, two of the branches of $l'$ are horizontal to $f$. Thus $D$ contains a loop. This is not possible as $D$ is a tree of rational curves. Hence none of the (-1) curves are in fibers of $f$. Hence all the (-1) curves in $D$ are horizontal to $f$. But there are exactly two components of $f$ which are horizontal to $D$. Hence $D$ cannot have more than two (-1) curves. Hence the number of cusps of $D$ is at most two.
\end{proof}

The number of contractible curves 
on a $\Q$-homology plane is of special interest. In \cite{gurjar88} Gurjar-Miyanishi and Gurjar-Parameswaran
\cite{gurjar95} proved that if the logarithmic Kodaira dimension of a $\Q$-homology plane is 0 or 1 then the surface 
has at most two contractible curves. In \cite{miyanishi92}, Miyanishi-Tsunoda proved the important result that a 
$\Q$-homology plane of general type does not contain any contractible curve.  We will use these results to prove:

\begin{lemma}
If $0 \leq \sk(\s-D-E) < 2$  then every $\C^*$-fibration on $\s-D-E$ is untwisted.
\end{lemma}

\begin{proof}
Note that given a $\C^*$-fibration $f$ on $\s-D-E$, we can extend $f$ to a $\C^*$-fibration on $\s-D$. That is to say, 
$E$ is contained in fibers of $f$. Indeed if $E$ is not contained in the fibers of $f$ then an irreducible component 
of $E$ is horizontal to $f$. Hence there are infinitely many contractible curves in $S-C$. This shows that 
$\sk(\s-D)=-\infty$. Which is contrary to hypothesis.\\

Case 1. C is uni-cuspidal.\\
We use the description of $D$ when $C$ is uni-cuspidal.\\

Under the hypothesis of the Lemma, if $C$ is uni-cuspidal then $C'$, the proper transform of $C$, is not a (-1)-curve. 
Indeed if it is a (-1)-curve then we can blow it down and map $D$ to a divisor with a non-contractible twig, which 
would imply that $\sk(\s-D)=-\infty$, as proved in \cite[Lemma 6.13]{fujita82}. This is contrary to the hypothesis.\\

Assume that $f$ is a twisted $\C^*$-fibration on $\s-D-E$. By Lemma 1.3, the $\C^*$-fibration $f$ 
is over $\A^1$. Further, all singular fibers of $f$ are irreducible and exactly one of them is an affine line (if taken
with reduced structure). In 
this case it is well known that the fiber at infinity is of the type [2,1,2].\\

Assume that $f$ has no base points on $D$. Then $B$ is the unique (-1)-curve in $D$ and is contained in the fiber at 
infinity. Let $\theta$ denote the 2-section of $f$ which is an irreducible component of $D$ adjacent to $B$. We know 
that $\theta$ can have branching number at most three. Let $\theta$ have branching number three. By $M_1$ and $M_2$ we 
denote the two branches of $\theta$ other than $B$. Note that $M_1$ and $M_2$ contain no (-1)-curves. If $M_1$ and $M_2$ 
are in separate fibers then both fibers will contain an $\A^1$ for the fibration over $\s-D-E$. This contradicts 
Lemma 1.3. Hence $M_1$ and $M_2$ lie in the same singular fiber of $f$. Let $F_0$ denote the fiber of 
$f$ containing both $M_1$ and $M_2$. It is clear that if $F_0 \cap (\s-D)$ contains an $\A^1$ then it contains at least 
one more irreducible component. This is not possible. Hence the branching number of $\theta$ is at most two. Therefore, 
the number of singular fibers of $f$ is at most two. This gives us an open subset of $\s-D-E$ which is a twisted 
$\C^*$-bundle over $\C^*$. Thus $\sk(\s-D-E)=0$. This is contrary to hypothesis.\\

We now prove that $f$ has no base points on $D$ to finish the proof in Case 1. If $f$ has a base point on $D$, it has 
to be unique. Indeed if $f$ has two base points on $D$ then $f$ is an untwisted fibration. Further, the base point of 
$f$ will have to lie on $B$ as $B$ is the unique (-1)-curve in $D$ with branching number three. Let $D'$ be the total 
transform of $D$ under the resolution of base locus. As the base point is on $B$, $D'$ has a unique (-1)-curve which 
is horizontal to $f$. But the fiber at infinity must also contain a (-1)-curve. This contradiction shows that $f$ has 
no base points on $D$ if $C$ is uni-cuspidal.\\

Case 2. C is bi-cuspidal.\\

We begin with a description of $D$ when $C$ is bi-cuspidal: By $C'$ we denote the proper transform of $C$. The 
branching number of $C'$ is two. It is adjacent to two (-1)-curves $H_1$ and $H_2$ both having branching number three. 
By $B_{11}$ and $B_{12}$ we denote the branches of $H_1$ other than $C'$. By $B_{22}$ and $B_{21}$ we denote the 
branches of $H_2$ other than $C'$. We can assume $B_{11}$ and $B_{22}$ to be linear. Also at least one of $B_{12}$ or $B_{11}$ and at least one of $B_{21}$ or $B_{22}$ has an irreducible component whose self intersection is less than -2. Both the branches of $C'$ 
can be blown down to smooth points.\\

Assume that $f$ has no base points on $D$. In this case there is a fiber at infinity of the type $[2,1,2]$. This 
is impossible by the description of $D$ above.\\

The fibration $f$ cannot have a base point on $D$: The base point of $f$ on $D$ will have to be unique and will have 
to lie on one of the (-1)-curves with branching number three. This forces the other (-1)-curve with branching number 
three to be in a singular fiber and have branching number three in the fiber. This is not possible. Hence $f$ has no 
base points on $D$ when $C$ is bi-cuspidal and the Lemma is proved.
\end{proof}

For our futher use, let us recall the notion of a \emph{rivet} (see \cite{fujita82}) in a 
$\C^*$-fibration. Let $f: \s-D \rightarrow C$ be a $C^*$-fibration and let $D_h$ denote the sum of components of $D$ which are 
horizontal to $f$. Let $F$ be a singular fiber of $f$. Then a connected component of $F \cap D$ is called a rivet if 
it meets $D_h$ in more than one point.\\  

\begin{lemma}
If $\sk(\s-D-E)=1$ then the $\C^*$-fibration on $\s-D$ has no base points on $D$.
\end{lemma}

\begin{proof}
Case 1. $C$ is uni-cuspidal.\\

By hypothesis there is a $\C^*$ fibration on $\s-D-E$ which we denote by $f$. If a component of $E$ is horizontal to 
$f$ then $S-C$ has infinitely many contractible curves, hence $\sk(\s-D)=-\infty$, which is  contrary to hypothesis. 
Thus $E$ is contained in the fibers of $f$. By Lemma 3.2, we know that $f$ is untwisted. We claim that $f$ extends to 
a $\p^1$-fibration on $\s$. Suppose that $f$ has a base point on $D$. Then at least one of the base points of $f$ lies 
on $B$. If $D'$ is the total transform of $D$ under the resolution of base locus, let $D'_{h_1}$ and $D'_{h_2}$ be the 
two horizontal components of $D'$. We may assume $D'_{h_1}$ to be a (-1)-curve with branching number at most 2 in $D'$. 
There is also a rivet $R$ in $D'$ for $f$. We claim that $R$ is not a full fiber of $f$. Indeed if $R$ were a full fiber of 
$f$, then $R$ would have to contain a (-1)-curve from $D'$. This would force $C'^2=-1$ and $C' \subset R$. If 
$C'^2=-1$ in $D'$ then $C'^2=-1$ in $D$. By Corollary 2.3, $\sk(\s-D)=-\infty$ which  
is contrary to hypothesis. Thus $R$ is not a full fiber of $f$. This shows that $f$ is an untwisted 
$\C^*$-fibration on $\s-D-E$ with base $\p^1$. Then by Lemma 1.3 all the fibers 
of $f$ are irreducible and one of them is $\A^1$ which meets $E$. We claim that 
$D'_{h_2}$ has branching number equal to the branching number of $D'_{h_1}$. If this is not true, then, one of the branches 
of $D'_{h_1}$ or $D'_{h_2}$ would lie in a fiber $F$ of $f$ such that $F \cap D'$ is connected. 
This would force $F \cap (\s-D-E)$ to be reducible, i.e. a union of $\C^*$ and an $\A^1$. This is a contradiction. Thus 
both $D_{h_1}$ and $D_{h_2}$ both have the same branching number. This gives us an open subset of $\s-D-E$ which is a 
$\C^*$ bundle over $\C^*$, so that $\sk(\s-D-E)=0$, which is not true. Thus $f$ has no base points on $D$.\\

Case 2. $C$ is bi-cuspidal.\\

Let, if possible, $f$ have base points on $D$ and let $D'$ be the total transform of $D$ under the resolution of base 
locus. Let $D'_{h_1}$ and $D'_{h_2}$ be the horizontal components of $D'$ to $f$. We may assume that $D'_{h_1}$ is 
a (-1)-curve and has branching number at most two in $D'$. While $D'_{h_2}$ has branching number at most three. As 
$D'$ is connected there is a rivet $R$ of $f$ in $D'$. Thus $R$ forms one of the branches of both $D'_{h_1}$ and 
$D'_{h_2}$.\\

We claim that $D'_{h_2}$ has branching number at least two. Indeed, if $R$ was the only branch of $D'_{h_2}$, then 
$f$ would have at most two singular fibers as the branching number of $D'_{h_1}$ is at most two. This will mean that 
there is an open subset of $\s-D-E$ which is a $\C^*$ bundle over $\C^*$. Thus, $\sk(\s-D-E)=0$. This is contrary 
to hypothesis. We denote by $M_2$ the branch of $D'_{h_2}$ other than $R$.\\

If the branching number of $D'_{h_1}$ is one, then there is an open subset of $\s-D-E$ which is a $\C^*$-bundle 
over $\C^*$. This implies that $\sk(\s-D-E)=0$, which is a contradiction. Hence we may assume that the branching 
number of $D'_{h_1}$ is two.\\

We further claim that $M_2$ has no (-1)-curve. If $M_2$ had a (-1)-curve in it then $M_2$ contains the proper 
transform of $C'$ or $H_1$ or $H_2$. If the proper transform of $C'$ is a (-1)-curve in $D$ then $C'$ will have 
to be in $R$ because the base points of $f$ can only lie on $H_1$ and $H_2$ and each $H_i$ has at most one base point. 
Thus $C'$ is not contained in $M_2$. If $H_1 \subset M_2$, then at least one of the irreducible components of $D'$ 
adjacent to $H_1$ is horizontal to $f$, because $H_1$ has branching number three in $D$. This makes $H_1$ a reduced 
component of the fiber of $f$ containing $M_2$. Hence $H_1$ cannot have branching number 2 in $M_2$. Thus $H_1$ is 
in $R$. Similarly, $H_2$ is not contained in $M_2$ This shows that $M_2$ has no (-1)-curves.\\

We have proved that $D'_{h_1}$ has branching number two. Let $M_1$ be the branch of $D'_{h_1}$ other than $R$. We 
claim that $M_1$ and $M_2$ lie in the same fiber of $f$. To see this, assume if possible that $M_1$ and $M_2$ are 
in different fibers of $f$, say $F_1$ and $F_2$. Then $F_1 \cap (\s-D-E)$ and $F_2 \cap (\s-D-E)$ are both reducible 
as $M_1$ and $M_2$ do not contain a (-1)-curve. This is impossible by Lemma 2.10 of \cite{miya91} as there is at most 
one reducible fiber to $f$. Thus $M_1$ and $M_2$ are in the same fiber of $f$ and yet again we get an open subset 
of $\s-D-E$ which is a $\C^*$-bundle over $\C^*$ which contradicts our hypothesis.\\

We may thus assume that $D'_{h_2}$ has a third branch $M_3$ such that $M_3 \subset F_3$. It is easy to see that 
$F_3$ contains a reduced $\C^*$ which will again produce an open set of $\s-D-E$ which is a $\C^*$ bundle over $\C^*$. 
This shows that $f$ has no base points on $D$.

\end{proof}

\begin{lemma}
If $\sk(\s-D-E)=1$ then there exists a smooth rational curve in $S^0$ passing through the cusps of $C$.
\end{lemma}

\begin{proof}

Case 1. $C$ is uni-csupidal.\\

By Lemma 3.3, $f$ extends to a $\p^1$-fibration on $\s$. Note that $f$ has a rivet $R$. We can show that $H$ is horizontal to $f$. 
Since $H$ is the unique (-1)-curve in $D$, we see that $R$ is not a full fiber of $f$ since it does not contain
any (-1)-curve. This implies that the base of the $C^*$ fibration $f$ is $\p^1$. Because $H$ has branching number 
three, the other horizontal component of $D$, say $H_2$, must also have branching number three. As $T_2$ is a linear 
chain, $H_2$ must lie in $T_3$. One of the branches of $H_2$ is $R$. The 
other two branches will have to lie in singular fibers of $f$ that contain $C'$ and $T_2$. By Lemma 1.3, we can say that $F_1=C' \cup l_1 \cup M_1$ and $F_2=T_2 \cup l_2 \cup M_2$ are full fibers of $f$ 
where $l_1$ and $l_2$ are (-1)-curves not in $D$ and $M_1$ and $M_2$ are branches of $H_2$ other than $R$. The fibers 
$F_1$ and $F_2$ have unique (-1)-curves hence $M_1$ and $M_2$ are linear chains of rational curves. Because $C'$ is 
an irreducible component of $F_1$, $M_1$ is a linear chain of (-2)-curves and $M_2$ is forced to be of the type 
$[m,2,...,2]$ (a linear chain of rational curves with self-intersections $-m,-2,..,-2$) where the (-m) curve in $M_2$ 
meets $H_2$. This is because $D-C'$ blows down to a smooth point. 
Using this and by looking at the contraction of $D-C'$ to a smooth point and keeping track of the image of $l_1$ we see that 
$l_1$ is the proper transform of a smooth rational curve $\Gamma$ on $S$ such that
$S-\Gamma$ contains the singular points of $S$.\\

Case 2. $C$ is bi-cuspidal.\\

We use the description of $D$ from Lemma 3.2, Case 2.\\

We first show that that $H_1$ and $H_2$ are the horizontal components of $f$:  If $H_1$ is in a fiber of $f$, 
then at least one of the irreducible components of $D$, adjacent to $H_1$ should be a cross-section to $f$ as $H_1$ 
has branching number three in $D$. But this means that $H_1$ is a reduced (-1)-curve in the fiber it belongs to. Hence 
it cannot have branching number two in the fiber. This forces two components of $D$ adjacent to $H_1$ to be horizontal 
to $f$. Thus $H_2$ is in a fiber of $f$ and at most one component of $H_2$ is a cross-section to $f$. Again, this is 
seen to be impossible.\\

As $H_1$ and $H_2$ are horizontal to $f$, $B_{ij}$ are contained in fibers of $f$. We claim that at least three members of the set $\eta = (B_{ij} )$ are linear chains of rational curves. We denote the fiber containing $C'$ by $F_0$. We claim that if $F_{ij}$ is a fiber containing $B_{ij}$, then $F_{ij}$ has to contain one more element from $\eta$. If not then there are two elements of $\eta$ which are the only elements of $\eta$ in the respective fibers. Let $F_{ij}$ and $F_{i'j'}$ be two such fibers. Both have a component not in $D$ which meets $H_1$ or $H_2$. This shows that $F_{ij}$ and $F_{i'j'}$ are both reducible as fibers of the $\C^*$ fibration. This is not possible by Lemma 1.3.\\

Assume that $B_{11}$ and $B_{22}$ are contained in the same fiber which we denote by $F_1$. Let $F_2$ denote the fiber containing $B_{12}$ and $B_{21}$. If $F_0$ is not a full fiber then both $F_1$ and $F_2$ are irreducible as fibers of the $\C^*$ fibration. Thus both have a unique (-1) curve $l_1$ and $l_2$ in them. Also, $F_1-l_1$ and $F_2-l_2$ are unions of two connected components and each component has a reduced curve of the fiber. Hence by an observation of Palka, $F_1$ and $F_2$ are linear chains of rational curves. Thus every element of $\eta$ is linear.\\

Next we assume that $F_0$ is a full fiber of $f$. In this case we may assume that $F_2$ is not irreducible as a fiber of the $\C^*$-fibration. The fibration $f$ on 
$\s$ has two singular fibers: $F_1=B_{11} \cup B_{22} \cup l$ and $F_2=B_{12} \cup B_{21} \cup l_1 \cup l_2 \cup E$, 
here $l$ is a (-1)-curve not contained in $D$. We have the following possibilities for $l_1$ and $l_2$ both of which 
are outside $D$.

\begin{enumerate}
\item $l_1$ and $l_2$ are affine lines and they intersect in a connected component of the exceptional locus of the resolution of a cyclic quotient singularity.
\item $l_1$ is an affine line and $l_2$ is a $\C^*$.
\end{enumerate}

In the first case, If $l_1$ and $l_2$ are affine lines meeting $E$ which is a linear chain of rationa curves, it can be seen, again using the observation above, that $B_{12}$ and $B_{21}$ are linear. Also $B_{11}$ and $B_{22}$ are linear. Hence, all members of $\eta$ are linear.\\

In the second case, we claim that $l_1$ is a (-1) curve. If not, then $l_2$ is the unique (-1) curve in the fiber and by Palka's observation $F_2$ is linear, a contradiction. Let $l_1$ be attached to $B_{12}$. We blow down $l_1$ and subsequent (-1) curves, in this process we do not end up blowing down the whole of $B_{12}$, for that would give us a non-contractible twig of the image of $D$. This would imply that $\sk(\s-D) =-\infty$. Which is contrary to hypothesis. This shows that after blowing down $l_1$ and subsequent (-1) curves, the image of $l_2$ is a (-1) curve. Thus $B_{21}$ is a linear chain of rational curves. Again $B_{11}$ and $B_{22}$ are linear because $F_1$ as a fiber of the $\C^*$-fibration is irreducible. This shows that at least 3 elements of $\eta$ are linear.\\

Thus we may assume that $F_1$ is linear and $B_{21}$ is linear. Thus $B_{21} \cup H_2 \cup B_{22}$ is linear. Then it is easy to see that $l$ is the proper transform of a  rational uni-cuspidal curve $L$ under the resolution of the cusps of $C$. Further since $l \cdot l =-1$, if $L'$ denotes the proper transform of $L$ under the minimal SNC resolution of the cusp of $L$, then $L' \cdot L' \geq 1$. Hence by Corollary 2.3, there is a smooth rational curve in $S^0$.

\end{proof}

To show that $E$ is irreducible. It is enough to prove the following
\begin{lemma}
Let $(S,C)$ be a singular $\Q$-homology plane pair. If $C$ is smooth then $E$ is irreducible.
\end{lemma}

\begin{proof}
 Recall that $S-C$ is affine. Hence $C^2 > 0$. We blow up a point on $C$ 
and infinitely near points until the proper transform of $C$ has self-intersection 0. Thus we have an $\A^1$-fibration 
$f$ on $\s-D$ where $D$ is the total transform of $C$ under the sequence of blow-ups. The fibration $f$ contains 
$E$ in its fibers. We know that the chain of (-2)-curves resulting from the blow ups lies in a fiber $F_0$ of $f$. 
We claim that $E$ is also contained in $F_0$. If some connected component of $E$ is not contained in $F_0$ then let 
$F_1$ be the fiber of $f$ that contains $E_1$, a connected component of $E$. The fiber $F_1$ has an irreducible component 
which meets the cross section of $f$ in $D$. This component is reduced, hence there is another irreducible component 
of $f$ which is a (-1)-curve and does not meet $D$. But $E$ has no (-1)-curves and $\s-D$ can have no complete curves 
except those in $E$. This is a contradiction. Hence $E$ is contained in $F_0$.\\

Next we claim that $E$ is irreducible. We know that $F_0 \cap D$ is a linear chain of (-2)-curves. Let $B$ denote 
the (-2)-curve which meets the cross section of the fibration. Let $\{E_i\}$ be the connected components of $E$ in $F_0$. 
Since $F_0$ is a tree of rational curves there is a unique irreducible curve  $l_i$ in $F_0$ that meets $E_i$ and 
$F_0 \cap D$ in one point each. Exactly one $l_i$ is a (-1)-curve because each $l_i$ intersects the linear chain of
(-2)-curves. We blow down this (-1)-curve and subsequent (-1)-curves till $B$ is mapped to a (-1)-curve under the 
blow downs. It is clear that if there are at least two $l_i$ then the branching number of $B$ is at least two. But $B$ is a 
reduced curve in the fiber, hence it meets only one other curve in $F_0$. This contradiction shows that there is exactly 
one $l_i$ and hence at most one connected component of $E$. To see that $E$ is irreducible we blow down the unique (-1)- 
curve in $F_0$ and subsequent (-1)-curves till $B$ is mapped to a (-1)-curve. As $B$ is reduced, there is another (-1)- 
curve in the image of $F_0$ which can only be the curve adjacent to the image of $B$. From this we see that $E$ is 
irreducible.\\

\end{proof}

To see that $S- \theta$ is $\Z$-homology plane we note that the unique (-1) curve in $F_0$ is the proper transform of a smooth rational curve which meets $C$ in one point. Thus the loop around $\theta$ in $S -\theta$ can be assumed to lie on an $\A^1$ in $\s$ which is topologically contractible. Thus the fundamental group of $\s- \theta$ is isomorphic to $\pi_1(\s)$ which is trivial because $\s$ is a rational surface (\cite{pradeep97} and \cite{gurjar97}). But the fundamental group of $S'$ and $\s$ are isomorphic as $E$ is a tree of rational curves. Hence $S-\theta$ is a $\Z$-homology plane.\\

This completes the proof of Corolary 1.2.\\

\section{Examples}

\subsection{An Example}
Consider a cubic $C$ on $\C\p^2$ with a single cusp. Let $T$ be the tangent line to $C$ at the cusp. Let $L$ be 
the tangent to $C$ at a smooth point of inflection. Let $p= T \cap L$, clearly $p$ is not on $C$. We blow up at $p$. 
We have a fibration with the proper transforms of $T$ and $L$ are fibers. Let $q$ be the point where the proper 
transform $L'$ of $L$ and $E$ meet. Here $E$ is the exceptional curve. We perform elementary transforms by blowing up 
at the point $q$ and blowing down the proper transform of $L$. One can see that  $\sk(\s-D)=1$ and $\sk(\s-D-E)=1$.

\bibliographystyle{plain}
\bibliography{biblio}
\vspace{2em}

Sagar Kolte\\
Indian Institute of Technology, Mumbai.\\
sagar@math.iitb.ac.in\\

\end{document}